\documentclass{article}
\usepackage{amssymb,latexsym}
\usepackage{amsmath}
\usepackage{amsthm}
\usepackage{amssymb}
\usepackage{graphicx}

\newcommand{\N}{\mathbb{N}}
\newcommand{\A}{\cal{A}}
\newcommand{\Lan}{\cal{L}}
\newcommand{\vpi}[1]{\varphi_{\mathrm{#1}}}
\newcommand{\vp}{\varphi}
\newcommand{\ubi}[1]{\mathbf{u}_{\mathrm{#1}}}
\newcommand{\ub}{\mathbf{u}}
\newcommand{\Lpro}{\mathrm{Lpro}}
\newcommand{\Rpro}{\mathrm{Rpro}}
\newcommand{\GL}[1]{\mathrm{GL}^{(#1)}}
\newcommand{\GR}[1]{\mathrm{GR}^{(#1)}}
\newcommand{\Le}{\mathrm{Lext}}
\newcommand{\Rex}{\mathrm{Rext}}
\newcommand{\wb}{\mathbf{w}}

\newtheorem{pro}{Proposition}
\newtheorem{lem}[pro]{Lemma}
\newtheorem{rem}[pro]{Remark}
\newtheorem{thm}[pro]{Theorem}
\newtheorem{dfn}[pro]{Definition}
\newtheorem{cor}[pro]{Corollary}

\newtheorem{exa}[pro]{Example}

\begin{document}

\begin{center}
	\textbf{Bispecial factors in circular non-pushy D0L languages}\\[0.5cm]
	Karel Klouda\\[0.5cm]
	\texttt{karel.klouda@fit.cvut.cz}\footnote{Faculty of Information Technology,
	Czech Technical University in Prague, Th\'{a}kurova 9, 160 00, Prague 6}\\
	\texttt{http://www.kloudak.eu/}
\end{center}

\begin{abstract}
We study bispecial factors in fixed points of morphisms. In
particular, we propose a simple method of how to find all bispecial words of
non-pushy circular D0L-systems. This method can be formulated as an algorithm.
Moreover, we prove that non-pushy circular D0L-systems are exactly those
with finite critical exponent.
\end{abstract}

\noindent
\small{\emph{keywords:} bispecial factors, circular D0L systems, non-pushy
D0L systems, critical exponent}

\section{Introduction}

Bispecial factors proved to be a powerful tool for better understanding
of complexity of aperiodic sequences of symbols from a finite set. One of the
most studied families of such sequences are fixed points of morphisms. In this
paper we present a method of how to describe the structure of all bispecial
factors in a given fixed point.

The method we describe here can be partially spotted in results of
several authors: It is a sort of inverse of the algorithm by Cassaigne from
paper \cite{Cassaigne1994} which is concerned by pattern avoidability. Very
similar approach was used in \cite{Avgustinovich1999} by Avgustinovich and Frid
and in \cite{Frid1998} by Frid to describe bispecial factors of biprefix
circular morphisms and marked uniform morphisms, respectively. Actually, the
fact that all bispecial factors in a fixed point can be generated as members of some
easily constructed sequences was noticed in many papers where factor
complexity was computed, see, e.g., \cite{Luca1988, Cassaigne1997}. In this paper we
formalize this approach and prove that it works for a very wide class of
morphisms, namely non-pushy and circular morphisms. Moreover, it seems that the
assumptions we need for proofs can be weakened or even omitted and the main
theorems remain true.

The paper is organized as follows. In the next section we introduce necessary
notation and notions and also explain the importance of bispacial factors.
Since it is easier to explain the main result using examples than to formulate
it as a theorem, we do so in
Section~\ref{sec:explain_main_result}. Section~\ref{sec:main_proofs} contains
proofs of the crucial theorems and in Section~\ref{sec:infinite_branches} we
explain how to use our results to identify immediately all infinite special
branches. In Section~\ref{sec:assumptions} we prove that non-pushy and circular
morphisms are exactly those whose fixed points have finite critical exponent.

\section{Preliminaries}

Let $\mathcal{A} = \{0,1,\ldots, n - 1\}, n \geq 2$, be a finite
alphabet of $n$ letters; if needed, we denote this particular $n$-letter
alphabet as $\mathcal{A}_n$. An \emph{infinite word} over the alphabet $\mathcal{A}$ is a
sequence $\ub = (u_i)_{i \geq 1}$ where $u_i \in \mathcal{A}$ for all $i \geq 1$. If $v = u_j u_{j+1}\cdots u_{j+n-1}$, $j,n \geq 1$, then $v$ is
said to be a \emph{factor} of $\ub$ of length $n$, the empty word $\epsilon$
is the factor of length $0$. The set of all finite words over $\mathcal{A}$ is
the free monoid $\mathcal{A}^*$, the set of nonempty finite words is denoted by
$\mathcal{A}^+ = \mathcal{A}^* \setminus \{\epsilon\}$.

A map $\vp : \mathcal{A}^* \to \mathcal{A}^*$ is called a \emph{morphism} if
$\vp(wv) = \vp(w)\vp(v)$ for every $w, v \in \mathcal{A}^*$. Any morphism $\vp$
is uniquely given by the set of images of letters $\vp(a), a \in \mathcal{A}$.
If all these images are nonempty words, the morphism is called
\emph{non-erasing}. A famous example of a morphism is the Thue-Morse morphism
$\vpi{TM}$ defined by
\begin{equation*}
	\begin{array}{rcl}
		\vpi{TM}(0)&=&01,\\
		\vpi{TM}(1)&=&10.
	\end{array}
\end{equation*}
This paper studies infinite \emph{fixed points} of morphisms: an infinite word
$w$ is a fixed point of a morphism $\vp$ if $\vp(w) = w$. If $\vp^\ell(w) = w$ for some positive $\ell$, $w$ is a \emph{periodic point} of
$\vp$. The fixed point of $\vpi{TM}$ beginning in the letter $0$ is the infinite
word
\begin{equation}\label{eq:def_uTM}
    \ubi{TM} = \lim_{n \to \infty} \vpi{TM}^n(0) = \vpi{TM}^\omega(0) =
    0110100110\cdots,
\end{equation}
which is called the \emph{Thue-Morse word}. 

An infinite word $\ub$ is \emph{aperiodic} if it is not
\emph{eventually periodic}, i.e., there are no finite words $v$ and $w$ such
that $\ub = v w w w w \cdots = v w^\omega$. If a word $u = v w$, then $v$ is a
\emph{prefix} of $u$ and $w$ is its \emph{suffix}. In this case we put
$(v)^{-1}u = w$ and $u(w)^{-1} = v$. Given a morphism $\vp$ on $\mathcal{A}$, if $\vp(a)$
is not a suffix of $\vp(b)$ for any distinct $a,b \in \mathcal{A}$, then
$\vp$ is said to be \emph{suffix-free}. \emph{Prefix-free} morphisms are defined analogously. 

The \emph{language of a fixed point} $\ub$ is the set of all its factors and is
denoted by $\mathcal{L}(\ub)$. When speaking about a morphism, we usually mean a
morphism together with its particular infinite fixed point. But a morphism can have more than one fixed point and
not all of them must have the same language (this is true if the
morphism is primitive). For instance, consider the
morphism $0 \mapsto 010, 1 \mapsto 11$: it has two fixed
points, one aperiodic starting in $0$ and one periodic starting in
$1$. Therefore, instead of speaking only about a morphism we
will always speak about a morphism and its particular infinite
fixed point. A well-established way of how to do so is to treat
a morphism and its fixed point as a D0L-system (see, e.g.,
\cite{Rozenberg1986} and \cite{Rozenberg1980}).
\begin{dfn}
    A triplet $G = (\mathcal{A},\vp,w)$ is called a \emph{D0L-system}, where
    $\mathcal{A}$ is an alphabet, $\vp$ a morphism on $\mathcal{A}$, and $w \in
    \mathcal{A}^+$ is an \emph{axiom}.
    The language of $G$ denoted by $\mathcal{L}(G)$ is the set
    of all factors of the words $\vp^n(w), n = 0,1,\ldots$

    If $\vp$ is non-erasing, then the system
    is called PD0L-system.
\end{dfn}
In what follows, when referring to a D0L-system, we always mean a
PD0L-system. In fact, for any D0L-system, it is possible to
construct its elementary (not simplifiable) version which is a
PD0L-system with an injective morphism \cite{Ehrenfeucht1978}
\cite{Seebold1985}.

Clearly, if $\vp(a) = av$ for some $a\in \mathcal{A}, v \in \mathcal{A}^+$, and if
$\vp$ is non-erasing, then the language of the D0L-system
$(\mathcal{A},\vp,a)$ is the language of the infinite fixed point
$\vp^\omega(a)$.

There are several tools which help us to study the structure of the language of
D0L-systems. We mention here two basic ones: the factor complexity and 
critical exponent. The factor complexity of a language is the function $C(n)$
which counts the number of factors of length $n$. The factor complexity is
usually obtained using the notion of \emph{special factors}.
\begin{dfn}
    Let $w$ be a factor of the language $\mathcal{L}(G)$ of a D0L-system $G$
    over $\mathcal{A}$. The set of \emph{left extensions} of $w$ is defined as
    $$
        \Le(w) = \{a \in \mathcal{A} : aw \in \mathcal{L}(G)\}.
    $$
    If $\#\Le(w) \geq 2$, then $w$ is said to be a \emph{left
    special (LS) factor} of $\mathcal{L}(G)$.

    In the analogous way we define the set of \emph{right extensions}
    $\Rex(w)$ and a \emph{right special (RS) factor}.  If $w$ is both left
    and right special, then it is called \emph{bispecial~(BS)}.
\end{dfn}
The connection between special factors and the factor complexity is described
in~\cite{Cassaigne1997}; the complete knowledge of LS, RS, or BS factors enables
to find the factor complexity.

The critical exponent is related to the repetitions in the language. Let
$w$ be a finite and nonempty word. Any finite prefix $v$ of $w^{\omega} =
www\cdots$ is a \emph{power} of $w$. We denote this by $v = w^r$, in words $v$
is \emph{$r$-power} of $w$, where $r = \frac{|v|}{|w|}$.  Further, we define the
\emph{index} of $w$ in a language $\mathcal{L}(G)$ of a D0L-system $G$ as 
$$ 
	\mathrm{ind}(w,G) = \sup\left\{ r \in \mathbb{Q} : w^r \in \mathcal{L}(G)
	\right\}. $$
And finally, the \emph{critical exponent} of the language $\mathcal{L}(G)$ is
the number
$$
	\sup\{ \mathrm{ind}(w) \mid w \in \mathcal{L}(G)\}.
$$
More details about critical exponent can be found, e.g.,
in~\cite{Krieger2007}. Examples of how knowledge of BS factors can help to
compute it for a fixed point of a morphism are in~\cite{Balkova2009}
and~\cite{Balkova2011}.

\section{Explaining the main result} \label{sec:explain_main_result}

Since the main result of this paper is a tool rather than a
theorem, we demonstrate it using example morphisms. The tool has
two ingredients. First one is a mapping that maps a BS factor to
another one and so, applied repetitively, it generates sequences
of BS factors. This mapping is defined by two directed labeled
graphs. The other ingredient is a finite set of BS factors such
that the sequences generated from them by the mapping cover all BS
factors in a given fixed point.

Let us consider the morphism $\vpi{E}$ defined by $0 \mapsto 012,
1 \mapsto 112, 2 \mapsto 102$ and the corresponding D0L-system
$({\mathcal{A}}_3, \vpi{E}, 0)$ with the fixed point $\ubi{E}$.
The factor $2112$ is LS and has left extensions $0$ and $1$. If we
apply the morphism $\vpi{E}$ on this structure -- both on the
factor and its two extensions --, the resulting factor
$\vpi{E}(2212)$ is no more LS since the respective extensions
$\vpi{E}(0) = 012$ and $\vpi{E}(1) = 112$ end in the same letter
$2$. In order to obtain a LS factor, we have to cut off the
longest common suffix of the new extensions, here it is $12$, and
append it to the beginning of the $\vpi{E}$-image of the factor:
the result is the LS factor $12\vpi{E}(2212)$ with left extensions
$0$ and $1$. We can proceed in the same manner and obtain another
LS factor $12\vpi{E}(12)\vpi{E}^2(2112)$ again with the same
extensions $0$ and $1$. Clearly, the same process works for RS
factors and right extensions.

Let us formalize what we did in the previous paragraph. Instead of
BS factors we will use a slightly different notion of \emph{BS
triplets} $((a,b), v, (c,d))$, where $v$ is a BS factor and
$(a,b)$ and $(c,d)$ are unordered pairs of its left and right
extensions, respectively. We assume that either $avc$ and $bvd$ or
$avd$ and $bvc$ are factors. Thus, $((0,1), 2112, (0,1))$ is a BS
triplet in $\ubi{E}$. In terms of the previous paragraph, we can
get another BS triplet from this one, namely $((0,1),
12\vpi{E}(2112),(0,1))$; we call this BS triplet the
\emph{$f$-image} of $((0,1), 2112, (0,1))$. The fact that left
extensions $(0,1)$ result again in extensions $(0,1)$ with
appending of $12$ can be represented as a directed edge from vertex
$(0,1)$ to vertex $(0,1)$ with label $12$. The edge corresponding
to the right extensions starts in $(0,1)$, ends again in $(0,1)$
and is labeled by the empty word. Applying this idea on all
possible pairs of left and right extensions gives us two directed
labeled graphs depicted in Figure~\ref{fig:GL_GR_E}. We call these graphs
\emph{graph of left and right prolongations}. With these graphs in hand, it is
easy to generate infinitely many BS triplets from a given starting one.
 \begin{figure}[ht]
        \begin{center}\includegraphics{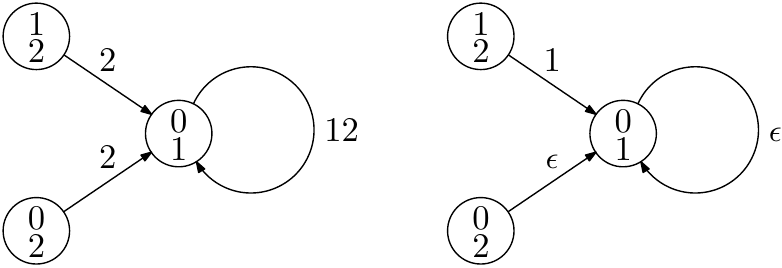}\end{center}
        \caption{The graphs defining the $f$-image for the morphism~$\vpi{E}$.} \label{fig:GL_GR_E}
    \end{figure}

The other ingredience of our method bears on the fact that all BS
triplets can be generated by taking repetitively $f$-image of only
finitely many \emph{initial} BS triplets. Initial BS triplets are
those which are not $f$-images of another BS triplet. For instance, $((0,1),
2112, (0,1))$ is not initial as it is the $f$-image of the BS
triplet $((0,1), 1, (0,2))$ which is initial. Later we will show
how to find all the initial factors for a given fixed point. For
the case of $\ubi{E}$, we have eight initial BS triplets:
\begin{center}
    \begin{tabular}{ c c c c }
      $((0,1),121,(0,1))$, & $((0,1),12,(0,1))$, & $((0,1),21,(0,1))$, & $((0,1),2,(0,1))$, \\
      $((1,2),1,(1,2))$,  & $((0,2),1,(1,2))$,  & $((0,2),1,(0,2))$, & $((1,2),0,(1,2))$.
    \end{tabular}
\end{center}

The vertices of the graphs from Figure~\ref{fig:GL_GR_E} are just
all pairs of distinct letters, but the situation is not that
simple for all morphisms. In fact, it happens for graph of left (right)
prolongations only if the respective morphism is suffix-free (prefix-free). This
case when the morphism is both prefix- and suffix-free has been already
solved in~\cite{Avgustinovich1999}, where not only describe the authors
all BS factors, but they also give a formula for the factor
complexity.

Let us consider the morphism $\vpi{S}$ defined by $0 \mapsto 0012,
1 \mapsto 2, 2 \mapsto 012$ and the corresponding D0L-system
$({\mathcal{A}}_3, \vpi{S}, 0)$.
Clearly, the morphism is not suffix-free. Let $v$ be a LS factor
with left extensions $(1,2)$. If we apply the morphism as above,
we have a problem: the longest common suffix of factors
$\vpi{S}(1) = 2$ and $\vpi{S}(2) = 012$ is $2$ and so we do not
know what are the left extensions of the LS factor $2\vpi{S}(v)$.
A solution is to consider left extensions longer than one letter, in such a
case we say left prolongation instead of left extension. Clearly, the factor
$1$ is always preceded by $0$. Hence, let us consider left extensions $(01,2)$. Now, $\vpi{S}(01) = 00122$
is no more a suffix of $\vpi{2} = 012$ and so we know the new left
extensions: again $(01,2)$. In this way we can construct a
complete graph defining the respective $f$-image, the result is in
Figure~\ref{fig:GL_GR_S} (the notation will be explained later).

To prove that a proper \emph{finite} set of pairs of left and
right extensions of arbitrary length always exists is not trivial.
It is not simple even to describe the properties such sets should
posses so that they define a correct $f$-image. We will call such
sets \emph{left} and \emph{right forky sets}, see
Definition~\ref{dfn:L_forky}. It may happen that a finite forky
set does not exist and so our method fails. Therefore we will
have to put some restriction on the D0L-systems considered: we
will assume that the systems are \emph{circular} and
\emph{non-pushy}. These notions are explained in the following
section.

Finally, we can now state the main result of this paper: Given a
circular non-pushy D0L-system with an aperiodic fixed point, there
exist finite left and right forky sets defining two directed
graphs and a finite set of initial BS triplets such that the
corresponding $f$-image applied repetitively on the initial
BS-triplets generates all BS factors.

\section{Forky sets and initial factors} \label{sec:main_proofs}

\subsection{Circular and non-pushy D0L-systems}

Any factor of a fixed point of a morphism $\vp$ can be decomposed
into (possibly incomplete) $\vp$-images of letters. For example,
in the case of $\vpi{E}$ we have: 01210 is a factor of
$\vpi{E}(0)\vpi{E}(2)$, i.e., 01210 is composed by
$\vpi{E}$-images of 0 and 2. We denote this using bars, i.e.,
$012|10$. The decomposition may not be unique. For instance $210$
is always decomposed as $2|10$ but we do not know whether 2 is a
suffix of $\vpi{E}(0)$ or $\vpi{E}(1)$ (it cannot be a suffix of
$\vpi{E}(2)$ since $22$ is not a factor of $\ubi{E}$). In the case
of the factor $1$, we do not even know where to place the bar if not
at all.

A factor can have more than one decomposition; however, if there
is a common bar for all these decompositions, this bar is called a
\emph{synchronizing point}. Coming back to our example, $210$ has
a synchronizing point between 2 and 10, formally we say that
$(2,10)$ is a synchronizing point of $210$.
\begin{dfn}[Cassaigne \cite{Cassaigne1994}]\label{dfn:synchro_poit_injective}
    Let $\vp$ be a morphism with a fixed point
    $\ub$, $\vp$ injective on $\mathcal{L}(u)$, and let $w$ be a factor of
    $\ub$. An ordered pair of factors $(w_1,w_2)$ is called a \emph{synchronizing point} of $w$ if $w = w_1w_2$ and
    $$
        \forall v_1, v_2 \in {\mathcal{A}}^{*}, ({v}_{1} {w} {v}_2 \in \vp(\mathcal{L}(\ub)) \Rightarrow  {v}_{1} {w}_{1} \in \vp(\mathcal{L}(\ub)) \text{\ and\ }  v_2w_2 \in
        \vp(\mathcal{L}(\ub))).
    $$
    We denote this by $w = w_1|_sw_2$.
\end{dfn}

\begin{dfn}\label{dfn:circular_D0L_system}
    A D0L-system $G = (\mathcal{A},\vp,w)$ is \emph{circular} if $\vp$ is injective on $\mathcal{L}(G)$ and if there
    exists $D \in \N$ such that any $v \in \mathcal{L}(G)$ with $|v| \geq D$ has at least one
    synchronizing point. This $D$ is called a \emph{synchronizing
    delay}.
\end{dfn}
Some examples of both circular and non-circular D0L-systems
follow.
\begin{exa}
    The system $G = (\mathcal{A}_2, \vpi{TM},0)$ is circular with a
    synchronizing delay 4. It is clear that any $w \in \mathcal{L}(G)$
    containing 00 or 11 has the synchronizing point $w =
    \cdots0|_s0\cdots$ or $w = \cdots1|_s1\cdots$. To see that, let us consider
    a word $w$ of length 4 not containing these two factors.
    Without loss of generality, assume that $w$ begins in 1, then $w = 1010$. This word can be
    decomposed into $\vpi{TM}(0)$ and $\vpi{TM}(1)$ in exactly two
    ways: $|10|10|$ and $1|01|0$. But the latter one is not admissible
    since it arises as the $\vpi{TM}$-image of $000$ which is not an alement of
    $\mathcal{L}(G)$.
\end{exa}

\begin{exa}
    The system $G = (\mathcal{A}_2,\vp,0)$, where $\vp(0) = 01, \vp(1) =
    11$, is not circular. Indeed, for all $n$ the word $1^n$ has no
    synchronizing point since it can be decomposed as $|11|11|\cdots$ and
    $1|11|11|\cdots$.
\end{exa}
This example is very simple since the respective infinite fixed point
$011111\cdots$ is eventually periodic. However, there are also
aperiodic non-circular systems.
\begin{exa}
    The system $G = (\mathcal{A}_3,\vp,0)$, where $\vp(0) = 010, \vp(1) =
    22$, $\vp(2) = 11$, is not circular. The argument is the same as in the previous
    example, since the words $1^n$ are for all $n \in \N$ elements of
    $\mathcal{L}(G)$. However, the infinite word $\vp^\omega(0) = 0102201011110102\cdots$ is aperiodic.
\end{exa}
One can notice that the languages in the both non-circular
examples contain an arbitrary power of 1. It is not just a
coincidence but a general rule.
\begin{thm}[Mignosi and S\'{e}\'{e}bold \cite{Mignosi1993}] \label{thm:mignosi_seebold}
    If a D0L-system is $k$-power-free (i.e., $\mathcal{L}(G)$ does not contain
    the $k$-power of any word) for some $k \geq 1$, then it is circular.
\end{thm}
Thus, non-circular fixed points must have infinite critical
exponent, in fact, they must
contain an unbounded power of some word.
\begin{thm}[Ehrenfeucht and Rozenberg \cite{Ehrenfeucht1983}]
\label{thm:strongly_repetitive}
    Given a D0L-system $G = (\mathcal{A}, \vp , w)$, if $\mathcal{L}(G)$ contains a
    $k$-power for all $k \in \N$, then $G$ is \emph{strongly repetitive}, i.e., there exists a nonempty $v \in L(G)$
    such that $v^\ell \in \mathcal{L}(G)$ for all $\ell \in \N$.
\end{thm}
We see that non-circular systems have very special properties.
Furthermore, the morphism of a non-circular system cannot be even
primitive. A morphism $\vp$ over $\mathcal{A}$ is primitive if there is $k \in
\N$ such that $\vp^k(a)$ contains $b$ for all $a,b \in \mathcal{A}$.
\begin{thm}[Moss\'{e} \cite{Mosse1996}]\label{thm:mosse}
    Any D0L-system $G = (\mathcal{A}, \vp , a)$ with $\vp$ injective on
    $\mathcal{G}$ and primitive is circular\footnote{In the article
    \cite{Mosse1996} the circular systems are called ``recognizable''.}.
\end{thm}

No matter how non-circular systems seem to be bizarre, there
is no known finite algorithm which would decide whether a given general
D0L-system is circular or not. Of course, if the respective
morphism is primitive, it is easy to prove it in finite steps. Later on we also
prove that if the system is non-pushy, then the circularity is equivalent to
repetitiveness which is decidable.
\begin{exa}
    An example of non-primitive but circular morphism is the one given by $0 \mapsto 0010, 1 \mapsto
    1$. This is the Chacon morphism~\cite{Chacon1969} and 5~is its synchronizing
    delay.
\end{exa}

\subsection{Non-pushy D0L-systems}

The following two definitions and the lemma are taken from
\cite{Ehrenfeucht1983}.
\begin{dfn}
    Let $G = (\mathcal{A},\vp,w)$ be a D0L-system. A letter $b \in \mathcal{A}$ has
    \emph{rank zero} if $\mathcal{L}(G_b)$, where $G_b =
    (\mathcal{A},\vp,b)$, is finite.
\end{dfn}
\begin{dfn}
    A D0L-system $G = (\mathcal{A},\vp,w)$ is \emph{pushy} if for all $n \in
    \N$ there exists $v \in \mathcal{L}(G)$ of length $n$ which is composed of
    letters that have rank zero; otherwise $G$ is \emph{non-pushy}. If $G$ is
    non-pushy, then $q(G)$ denotes 
    $$
        q(G) = \max\{|v| \mid v \in \mathcal{L}(G) \text{ is composed of letters
        that have rank zero}\}. $$
\end{dfn}
\begin{lem}\label{lem:pushy_systems} \
    \begin{enumerate}
        \item It is decidable whether or not an arbitrary D0L-system
        is pushy.
        \item If $G$ is pushy, then $\mathcal{L}(G)$ is strongly repetitive (see Theorem~\ref{thm:strongly_repetitive}).
        \item If $G$ is non-pushy, then $q(G)$ is effectively computable.
        \item It is decidable whether or not an arbitrary D0L-system is strongly
        repetitive.
    \end{enumerate}
\end{lem}
\begin{cor}[Krieger \cite{Krieger2007}] \label{cor:constant_C}
    Let $G = (\mathcal{A},\vp,a), a \in \mathcal{A},$ be a non-pushy D0L-system
    and let $\ub = \vp^\omega(a)$ be an infinite fixed point of $\vp$. Then
    there exists a non-erasing morphism $\vp'$ and an effectively computable $C \in \N$ such that $\ub =
    (\vp')^\omega(a)$ and for all $v \in \mathcal{L}(G)$ with $|\vp'(v)| =
    |v|$ we have $|v| < C$.
\end{cor}
This means that for any word $v$ of length at least $C$ we have
$|\vp(v)| \geq |v| + 1$. More generally, if $|v| \geq KC$ then
$|\vp(v)| \geq |v| + K$. We will use this in the proof of the main
theorem of the following subsection.

In the sequel, we always suppose that the $G = (\mathcal{A},\vp,
a)$ is such that $\vp'$ can be taken equal to $\vp$ (in fact
$\vp'$ is just a power of $\vp$, see the proof
in~\cite{Krieger2007}). This is without loss of generality since
the language is the same.

\subsection{Forky sets}

Our aim is to define properly the notion of $f$-image introduced in
Section~\ref{sec:explain_main_result}. As explained, we need to have two directed labeled graphs defined on unordered pairs of
left and right prolongations. Since left and right extensions usually
refer to letters and the vertices of our graphs might be pairs of
words, we give the following definition.
\begin{dfn}
    Let $\ub$ be an infinite word and $w$ its factor. The set
    of \emph{left prolongations} of $w$ is the set
    $$
        \Lpro(w) = \{v \in \mathcal{A}^+ \mid vw \in \mathcal{L}(\ub)\}.
    $$
    In an analogous way we define the set of \emph{right prolongations} $\Rpro(w)$.
\end{dfn}
The sets $\Lpro(w)$ and $\Rpro(w)$ are, in general, infinite.

Our aim is to specify a suitable finite sets $\mathcal{B}_L$ and $\mathcal{B}_R$
of (unordered) pairs of left and right prolongations such that it
allows to define correctly an $f$-image of all triplets
$((w_1,w_2), v, (w_3,w_4))$, where $v$ is a BS factor, $(w_1,
w_2)$ a pair of its left and $(w_3, w_4)$ a pair of its right
prolongations from $\mathcal{B}_L$ and $\mathcal{B}_R$, respectively. The
$f$-image defined by the sets $\mathcal{B}_L$ and $\mathcal{B}_R$ are to be
defined as a BS triplet $((w'_1,w'_2), v', (w'_3,w'_4))$, where $(w'_1,w'_2)$
and $(w'_3,w'_4)$ are again in $\mathcal{B}_L$ and $\mathcal{B}_R$ and the
factor $v' = f_L(w'_1,w'_2)\vp(v)f_R(w'_3,w'_4)$ is BS. The mappings
$f_L$ and $f_R$ are defined as follows.
\begin{dfn}\label{dfn:f_L_f_R_intro}
    Let $\vp$ be a morphism over $\mathcal{A}$ and let $(v_1, v_2)$ be an unordered pair of words from $\mathcal{A}^+$.
    We define
    \begin{eqnarray*}
      f_L(v_1,v_2) &=& \text{the\ longest\ common\ suffix\ of\ } \vp(v_1) \text{\ and\ } \vp(v_2), \\
      f_R(v_1,v_2) &=& \text{the\ longest\ common\ prefix\ of\ } \vp(v_1) \text{\ and\ } \vp(v_2).
    \end{eqnarray*}
\end{dfn}
The purpose of the following definitions is just to describe
``good'' choices of $\mathcal{B}_L$ and $\mathcal{B}_R$.
\begin{dfn}
    Let $(w_1,w_2)$ and $(v_1,v_2)$ be unordered pairs of
    words. We say that
    \begin{itemize}
        \item[(i)] $(w_1,w_2)$ is a \emph{prefix} (\emph{suffix}) of
                    $(v_1,v_2)$ if either $w_1$ is a prefix (suffix) of $v_1$ and $w_2$ of
                    $v_2$, or $w_1$ is a prefix (suffix) of $v_2$ and $w_2$ of
                    $v_1$;
        \item[(ii)] $(w_1,w_2)$ and $(v_1,v_2)$ are \emph{L-aligned}
                    if
                    $$
                    (v_1 = uw_1 \text{\ or\ }  w_1 = uv_1) \quad \text{and} \quad ( v_2 = u'w_2  \text{\ or\ } w_2 = u'v_2)
                    $$
                    or
                    $$
                    (v_1 = uw_2 \text{\ or\ } w_2 = uv_1) \quad \text{and} \quad (v_2 = u'w_1 \text{\ or\ } w_1 = u'v_2)
                    $$
                    for some words $u,u'$.
    \end{itemize}
    Analogously, we define pairs which are \emph{R-aligned}.
\end{dfn}
\begin{exa}
    The pairs $(01,0)$ and $(001,10)$ are L-aligned,
    while $(01,0)$ and $(011,10)$ are not L-aligned.
    Schematically, the notion of L-aligned pairs of words is depicted in
    Figure~\ref{fig:L_aligned_pairs}.
    \begin{figure}[ht]
        \begin{center}\includegraphics{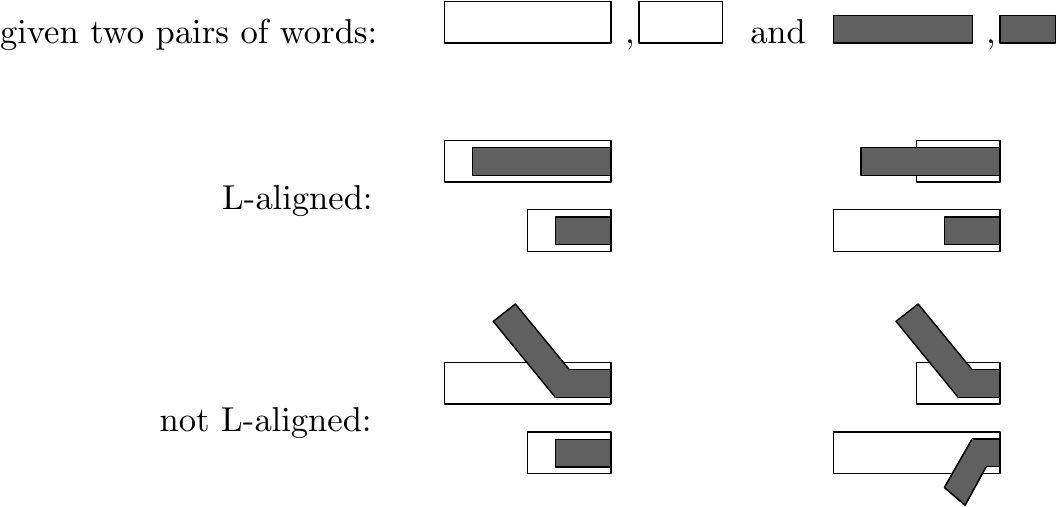}\end{center}
        \caption{L-aligned and not L-aligned pairs of words.} \label{fig:L_aligned_pairs}
    \end{figure}
\end{exa}
\begin{dfn} \label{dfn:L_forky}
    Let $\vp$ be a morphism with a fixed point $\ub$.
    A finite set $\mathcal{B}_L$ of unordered pairs $(w_1,w_2)$ of nonempty factors of
    $\ub$ is called \emph{L-forky} if all the following conditions
    are satisfied:
    \begin{itemize}
        \item[(i)] the last letters of $w_1$ and $w_2$ are different for all $(w_1, w_2) \in \mathcal{B}_L$,
        \item[(ii)] no distinct pairs $(w_1, w_2)$ and $(w'_1, w'_2)$ from $\mathcal{B}_L$ are L-aligned,
        \item[(iii)] for any $v_1,v_2 \in
        \mathcal{L}(\ub) \setminus \{\epsilon\}$ with distinct last letters there exists $(w_1,w_2) \in \mathcal{B}_L$ such that
        $(w_1,w_2)$ and $(v_1,v_2)$ are L-aligned,
        \item[(iv)] for any $(w_1, w_2) \in \mathcal{B}_L$ there exists $(w'_1, w'_2) \in \mathcal{B}_L$
        such that
        $$
            (w'_1f_L(w_1,w_2), w'_2f_L(w_1,w_2))
        $$
        is a suffix of $(\vp(w_1), \vp(w_2))$.
    \end{itemize}

    Analogously we define an \emph{R-forky} set.
\end{dfn}
Since the definition may look a bit intricate, we now comment on
all the conditions. Condition $(i)$ says that $w_1$ and $w_2$ are
left prolongations of LS factors (note that all pairs of words
$w_1,w_2 \in \mathcal{L}(\ub)$ are prolongations of the empty
word, i.e., elements of $\Lpro(\epsilon)$). Condition~$(ii)$ is
required to avoid redundancy in $\mathcal{B}_L$. Condition~$(iii)$
ensures that any two left prolongations of any LS factor are
included in $\mathcal{B}_L$ in the following sense: if we prolong
or shorten them in a certain way we obtain a pair from
$\mathcal{B}_L$. And, finally, Condition~$(iv)$ is there because
of the definition of the $f$-image: we want to be able to apply it
repetitively. Note that due to $(ii)$ and $(iii)$ the pair $(w'_1,
w'_2)$ from $(iv)$ is uniquely given. Note also that if $(i)$ is
satisfied and the words from all the pairs of $\mathcal{B}_L$
are of the same length, then $(ii)$ and $(iii)$ are satisfied
automatically.
\begin{exa} \label{exa:vp_S_set_B}
    Consider the morphism $\vpi{S}$ from
    Section~\ref{sec:explain_main_result} deifned by $0 \mapsto 0012,
	1 \mapsto 2, 2 \mapsto 012$. This morphism is injective
    and primitive and so, by Theorem~\ref{thm:mosse}, the respective D0L-system is circular.
    One can easily prove that 3 is a synchronizing delay
    (note that all factors containing $2$ has a synchronizing point $\cdots2|\cdots$.)

    Since $\vpi{S}$ is prefix-free, we get that the set
    $$
        \mathcal{B}_R = \{(0,1), (0,2), (1,2)\}
    $$
    is R-forky. However, this set is not L-forky:
    Condition $(iv)$ is not satisfied for any pair since $\vpi{S}(1)$ is a suffix of $\vpi{S}(2)$
    which is a suffix of $\vpi{S}(0)$. To remedy this, we
    consider left prolongations one letter longer which are ending in $1$ and
    $2$. Since the list of all factors of length 2 reads
    $$
        00,01,12,20,22
    $$
    the new pairs are
    \begin{equation}\label{eq:exa_vpS_list_Lpro}
        (0,01), (0,12), (0,22), (2,01).
    \end{equation}
    For these pairs conditions $(i)$ -- $(iii)$ are again satisfied.
    But $(iv)$ is not satisfied for $(0,12)$ since $f_L(0,12) =
    012$ and $(\vpi{S}(0)(012)^{-1}, \vpi{S}(12)(012)^{-1}) =
    (0,2)$ has no suffix in list~\eqref{eq:exa_vpS_list_Lpro}. Hence,
    we have to prolong $12$ again. There is only one possibility,
    namely $012$. The resulting set
    $$
        \mathcal{B}_L = \{(0,01), (0,012), (0,22), (2,01)\}
    $$
    is then L-forky since now  we get $(\vpi{S}(0)(012)^{-1}, \vpi{S}(012)(012)^{-1}) =
    (0,00122)$ with a suffix $(0,22) \in \mathcal{B}_L$.
\end{exa}
\begin{thm} \label{thm:existence_of_the_graph}
    Let $\vp$ be a morphism on $\mathcal{A}$ with a fixed point $\ub =
    \vp^\omega(a)$. If $(\mathcal{A},\vp, a)$ is circular non-pushy system, then
    it has L-forky and R-forky sets.
\end{thm}
\begin{proof}
    Set $M = DC$, where $D$ is a synchronizing delay and
    $C$ is the constant from Corollary~\ref{cor:constant_C}.
    Define
    \begin{multline}
        \mathcal{B}_L = \{(w_1,w_2) : w_1,w_2 \in \mathcal{L}(\ub), |w_1| =
        |w_2| = M, \\ \text{ and the last letters of } w_1 \text{ and }
        w_2
        \text{ are different}
        \}.
    \end{multline}
    We claim that $\mathcal{B}_L$ is L-forky. Conditions $(i)$ -- $(iii)$
    from Definition~\ref{dfn:L_forky} are trivially fulfilled. It remains to prove $(iv)$.

    It is clear that we must have $|f_L(w_1,w_2)| \leq
    D$ for any $(w_1,w_2)$ since $f_L(w_1,w_2)$ without the last letter does not have any
    synchronizing point. Now, it suffices to realize that for any
    $w$ of length $M$ we have $|\vp(w)| \geq |w| + D = M + D$ and so
    $(\vp(w_1)(f_L(w_1,w_2))^{-1}, \vp(w_2)(f_L(w_1,w_2))^{-1})$
    has a suffix in $\mathcal{B}_L$. The proof of existence of an R-forky set is
    perfectly the same.
\end{proof}
This proof does not give us a general guideline how to construct
L-forky and R-forky sets since, as we have recalled earlier, we do
not have an effective algorithm computing a synchronizing delay
for a general morphism. Moreover, the L-forky set constructed in
the proof is usually too huge. As in the case of our example
morphism $\vpi{S}$ from Example~\ref{exa:vp_S_set_B} (for this
morphism $C = 2$ and $D = 3$), there usually exists a smaller
L-forky set.
\begin{rem}\label{rem:important_Krieger}
    The techniques of the proof of
    Theorem~\ref{thm:existence_of_the_graph} are the same as those
    of the proof of Lemma~11 in~\cite{Krieger2007}. This Lemma,
    however, is concerned with the notion of critical exponent.
\end{rem}
\begin{dfn}
    Let $\vp$ be a morphism with a fixed point $\ub$ and let $\mathcal{B}_L$ be an L-forky set.
    We define the directed labeled \emph{graph of left prolongations} $\GL{\mathcal{B}_L}_\vp$ as follows:
    \begin{itemize}
        \item[(i)] the set of vertices is $\mathcal{B}_L$,
        \item[(ii)] there is an edge from $(w_1,w_2)$ to $(w_3,w_4)$
        if $(w_3f_L(w_1,w_2),w_4f_L(w_1,w_2))$ is a suffix
        of $(\vp(w_1),\vp(w_2))$. The label of this edge is
        $f_L(w_1,w_2)$.
    \end{itemize}
    In the same manner we define the \emph{graph of right prolongations} $\GR{\mathcal{B}_R}_\vp$.
\end{dfn}
As a straightforward consequence of the definition of forky sets
(especially of Condition $(iv)$) we have the following property of
the graphs.
\begin{lem}
    Each vertex in a graph of left and right prolongations has
    its out-degree equal to one.

    Consequently, any long enough path in the graph ends in a
    cycle and any component contains exactly one cycle.
\end{lem}
\begin{exa}
    The graphs $\GL{\mathcal{B}_L}_{\vpi{S}}$ and $\GR{\mathcal{B}_R}_{\vpi{S}}$ for $\vpi{S}$ and for the
    sets $\mathcal{B}_L$ and $\mathcal{B}_R$ from Example~\ref{exa:vp_S_set_B}
    are in Figure~\ref{fig:GL_GR_S}.
    \begin{figure}[ht]
        \begin{center}\includegraphics{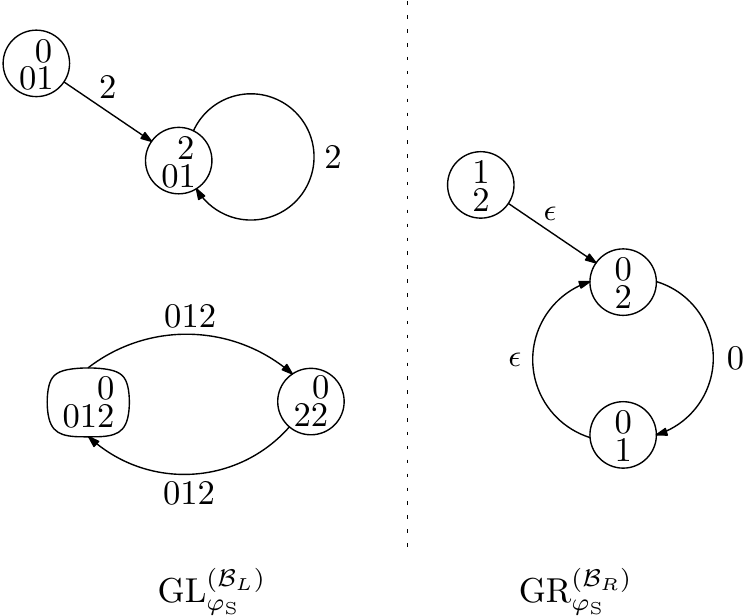}\end{center}
        \caption{The graphs $\GL{\mathcal{B}_L}$ and $\GR{\mathcal{B}_R}$ for the morphism~$\vpi{S}$.} \label{fig:GL_GR_S}
    \end{figure}
\end{exa}
\begin{dfn}
    Let $\vp$ be a morphism on $\A$ with a fixed point $\ub$
    and let $\mathcal{B}_L$ and $\mathcal{B}_R$ be L-forky and R-forky sets, respectively.
    A triplet $((w_1,w_2), v, (w_3,w_4))$ is called a \emph{bispecial (BS) triplet} in $\ub$
    if $(w_1,w_2) \in \mathcal{B}_L$, $(w_3,w_4) \in \mathcal{B}_R$ and $w_1vw_3, w_2vw_4 \in \mathcal{L}(\ub)$ or $w_1vw_4, w_2vw_3 \in \mathcal{L}(\ub)$.
\end{dfn}
\begin{lem}
    Let $\vp$ be a morphism with a fixed point $\ub$, let $\mathcal{B}_L$ be an
    L-forky and $\mathcal{B}_R$ an R-forky set and let $\mathcal{T} = ((w_1,w_2), v,
    (w_3,w_4))$ be a bispecial triplet of $\ub$. Let us denote by
    \begin{itemize}
        \item[(i)] $g_L(w_1,w_2)$ the end of the edge of $\GL{\mathcal{B}_L}_\vp$ starting in $(w_1,w_2)$,
        \item[(ii)] $g_R(w_3,w_4)$ the end of the edge of $\GR{\mathcal{B}_R}_\vp$ starting in $(w_3,w_4)$.
    \end{itemize}
    Then
    $$
        \mathcal{T}' = (g_L(w_1,w_2), f_L(w_1,w_2)\vp(v)f_R(w_3,w_4), g_R(w_3,w_4))
    $$
    is also a bispecial triplet of $\ub$.
\end{lem}
\begin{dfn}\label{dfn:f_mapping_for_forky_sets}
    Denote $\mathcal{B} = (\mathcal{B}_L, \mathcal{B}_R)$. The bispecial triplet $\mathcal{T}'$ from the previous lemma is
    called the \emph{$f_{\mathcal{B}}$-image} of a bispecial triplet $\mathcal{T} =((w_1,w_2), v, (w_3,w_4))$.
\end{dfn}
\begin{exa}
    Consider again the morphism $\vpi{S}$. Let $\mathcal{B} = (\mathcal{B}_L,\mathcal{B}_R)$,
    where $\mathcal{B}_L$ and $\mathcal{B}_R$ are those from
    Example~\ref{exa:vp_S_set_B}. Then $((0,012),0,(0,1))$ is a bispecial
    triplet since both $001$ and $01200$ are factors. Its $f_\mathcal{B}$-image reads
    $((0,22),0120012,(0,2))$ for we have $g_L(0,012) =
    (0,22)$, $f_L(0,012) = 012$, $g_R(0,1) = (0,2)$, and $f_R(0,1) = \epsilon$.
\end{exa}
Condition~$(iv)$ from Definition~\ref{dfn:L_forky} of forky sets
allows us to get a compact formula for the
$(f_{\mathcal{B}})^n$-image, i.e., $f_\mathcal{B}$-image applied repetetively
$n$ times.
\begin{lem}\label{lem:compact_form_of_f_n_image}
    Let $\vp$ be a morphism and let $((w_1,w_2),v,(w_3,w_4))$ be a bispecial triplet for some
    forky sets $\mathcal{B} = (\mathcal{B}_L, \mathcal{B}_R)$. Then for all $n \in \N$ it
    holds that its $(f_{\mathcal{B}})^n$-image equals
    $$
        (g_L^n(w_1,w_2),f_L(\vp^{n-1}(w_1),\vp^{n-1}(w_2))\vp^n(v)f_R(\vp^{n-1}(w_3),\vp^{n-1}(w_4)),g_R^n(w_3,w_4)).
        $$
\end{lem}

\subsection{Initial BS factors}

From the previous subsection, we know how to get a sequence of BS
factors from some starting one: we just apply
$f_\mathcal{B}$-image repetitively. The goal of the present
subsection is to prove that for circular systems there exists a
finite set of initial BS factors (triplets) such that any other BS factor is
an $(f_\mathcal{B})^n$-image of one of them.
\begin{dfn}
    Let $\vp$ be a morphism injective
    on $\mathcal{L}(\ub)$, where $\ub$ is its fixed point, let $\mathcal{B}_L$ and $\mathcal{B}_R$ be L- and R-forky sets,
    and $\mathcal{T} = ((w_1,w_2), v, (w_3,w_4))$ a bispecial triplet.
    Assume, without loss of generality, that $w_1vw_3, w_2vw_4 \in \mathcal{L}(\ub)$.
    An ordered pair of factors $(v_1,v_2)$ is called a \emph{BS-synchronizing point} of $\mathcal{T}$
    if $v = v_1v_2$ and
    \begin{multline*}
        \forall u_1, u_2, u_3, u_4 \in \mathcal{A}^*, (u_1 w_1 v w_3 u_3, u_2 w_2 v w_4 u_4 \in \vp(\mathcal{L}(\ub)) \Rightarrow \\
        u_1 w_2 v_1,u_2 w_2 v_1, v_2 w_3 u_3,v_2 w_4 u_4  \in \vp(\mathcal{L}(\ub)).
    \end{multline*}
    We denote this by $v = v_1|_{bs}v_2$.
\end{dfn}
The notion of BS-synchronizing point is weaker than the one of
synchronizing point: it holds that if $v = v_1|_s v_2$, then $v =
v_1 |_{bs} v_2$. The other direction is not true as it follows
from this example:
\begin{exa}
    Given a morphism $0 \mapsto 010, 1 \mapsto 210, 2 \mapsto
    220$, the factor $0$ has no synchronizing point. On the other
    hand, if we take it as the bispecial triplet $((1,2),0,(0,2))$
    it has the BS-synchronizing point $0|_{bs}$.
\end{exa}
\begin{dfn}
    Let $\vp$ be a morphism with a fixed point $\ub$ which is injective
    on $\mathcal{L}(\ub)$. A bispecial triplet $\mathcal{T} = ((w_1,w_2), v, (w_3,w_4))$
    is said to be \emph{initial} if it does not have any
    BS-synchronizing point.
\end{dfn}
\begin{dfn}\label{dfn:synchro_terminology}
    Let $\mathcal{T} = ((w_1,w_2), v, (w_3,w_4))$ be a bispecial triplet which is not initial and let $(v_1,v_2), (v_3,v_4), \ldots,
    (v_{2m-1},v_{2m})$ be all its BS-synchronizing points such
    that $|v_1| < |v_3| < \cdots < |v_{2m-1}|$. The factor
    $v_1$ is said to be the \emph{non-synchronized prefix}, $v_{2m}$ the \emph{non-synchronized
    suffix} and the factor $(v_1)^{-1}v(v_{2m})^{-1}$ is called the \emph{synchronized factor} of $\mathcal{T}$.
\end{dfn}
Now we state the main theorem of this section.
\begin{thm}
     Let $(\mathcal{A},\vp, a)$ be a circular non-pushy D0L-system,
     $\mathcal{B}_L$ and  $\mathcal{B}_R$ its L-forky and R-forky set, and $\ub
     = \vp^\omega(a)$ infinite. Then there exists a finite set $\mathcal{I}$  of
     bispecial triplets such that for any bispecial factor $v$ there exist a 
     bispecial triplet $\mathcal{T} \in \mathcal{I}$ and $n \in \N$ such  that
     $((w_1,w_2),v,(w_3,w_4)) = (f_{\mathcal{B}})^n(\mathcal{T})$ for some 
     $(w_1,w_2) \in \mathcal{B}_L$ and $(w_3,w_4) \in \mathcal{B}_R$.
\end{thm}

\begin{proof}

    Let $\mathcal{I}$ be the set of initial bispecial triplets. The
    finiteness of $\mathcal{I}$ is a direct consequence of the
    definition of circularity: elements of $\mathcal{I}$ cannot be
    longer than the synchronizing delay. The rest of the statement
    follows from the fact that any non-initial triplet has at least
    one $f_{\mathcal{B}}$-preimage.

    To prove this, it suffices to realize that the synchronized
    factor of $((w_1,w_2),v,(w_3,w_4))$ has unique $\vp$-preimage $v'$ (possibly the empty word) and
    that the non-synchronized prefix (resp. suffix) must be equal
    to $f_L(w'_1,w'_2)$ for some $(w'_1,w'_2) \in \mathcal{B}_L$
    (resp. to $f_R(w'_3,w'_4)$ for some $(w'_3,w'_4) \in
    \mathcal{B}_R$).

\end{proof}
\begin{exa}
    We now find the set $\mathcal{I}$ for $\ubi{S}$ with $\mathcal{B} =
    (\mathcal{B}_L,\mathcal{B}_R)$, where $\mathcal{B}_L$ and $\mathcal{B}_R$ are from
    Example~\ref{exa:vp_S_set_B}, the graphs of prolongations are in Figure~\ref{fig:GL_GR_S}. The only BS factors without
    synchronizing points are $\epsilon$ and $0$ since any other factor
    contains the letter~$2$ (and hence it has a synchronizing point) or is not
    BS. It remains to find all corresponding triplets:
    \begin{center}
    \begin{tabular}{ c c c c }
      $((0,01),\epsilon,(1,2))$, & $((0,01),\epsilon,(0,2))$, & $((0,012),\epsilon,(0,1))$, & $((0,012),\epsilon,(0,2))$, \\
      $((0,012),\epsilon,(1,2))$, & $((0,22),\epsilon,(0,1))$, & $((2,01),\epsilon,(0,2))$, & $((0,012),0,(0,1))$.
    \end{tabular}\\[2mm]
    \end{center}
    Since we are usually interested in nonempty BS factors, we can
    replace the bispecial triplets containing $\epsilon$ with their
    $f_\mathcal{B}$-images and get:
    \begin{center}
    \begin{tabular}{ c c c }
      $((2,01),2,(0,2))$, & $((2,01),20,(0,1))$, & $((0,22),012,(0,2))$, \\
      $((0,22),0120,(0,1))$, & $((0,012),012,(0,2))$, & $((0,012),1,(0,1))$.
    \end{tabular}\\[2mm]
    \end{center}
    There are only 6 bispecial triplets since $((0,012),\epsilon,(0,1))$ and
    $((0,012),\epsilon,(1,2))$ have the same $f_\mathcal{B}$-image and so do $((0,01),\epsilon,(0,2))$ and
    $((2,01),\epsilon,(0,2))$.
\end{exa}

\section{Infinite special branches} \label{sec:infinite_branches}

In the preceding section we have described a tool allowing us to
find all BS factors. It requires some effort to construct the
graphs $\mathrm{GL}$ and $\mathrm{GR}$ and to find all initial
bispecial triplets, but it can be done by an algorithm. However,
even if we have all these necessities in hand, it may be still a
long way to the complete knowledge of the structure of all BS
factors. Nevertheless, there is a class of special factors which
can be identified directly from the graphs $\mathrm{GL}$ and
$\mathrm{GL}$, namely, the prefixes (or suffixes) of the so-called
infinite LS (or RS) branches.
\begin{dfn}
    An infinite word $\mathbf{w}$ is an \emph{infinite LS
    branch} of an infinite word $\ub$ if each prefix of $\mathbf{w}$ is~a~LS
    factor of $\ub$. We put
    $$
        \Le(\wb) = \bigcap_{v \text{ prefix of } \wb}\Le(v).
    $$
    Infinite RS branches are defined in the same manner, only that they are
    infinite to the right.
\end{dfn}
Here are some (almost) obvious statements on infinite special
branches in an infinite word:
\begin{pro} \label{pro:inf_LS_branch} \
    \begin{itemize}
        \item[(i)]  If $\ub$ is eventually periodic, then there is
                    no infinite LS branch of $\ub$,
        \item[(ii)] if $\ub$ is aperiodic, then there exists at
        least one infinite LS branch of $\ub$,
        \item[(iii)] if $\ub$ is a fixed point of a primitive
                     morphism, then the number of infinite LS
                     branches is bounded.
    \end{itemize}
\end{pro}
\begin{proof}
Item $(i)$ is obvious, $(iii)$ is a direct consequence of the
fact that the first difference of complexity is
bounded~\cite{Cassaigne1996}. The proof of item $(ii)$ is due to
the famous K\"{o}nig's infinity lemma~\cite{Konig1936} applied on
sets $V_1, V_2, \ldots$, where the set $V_k$ comprises all LS
factors of length $k$ and where $v_1 \in V_i$ is connected by an
edge with $v_2 \in V_{i+1}$ if $v_1$ is prefix of $v_2$.
\end{proof}

Imagine now that we have an L-forky set $\mathcal{B}_L$ and an
infinite LS branch $\wb$. There must exist $(v_1,v_2) \in
\mathcal{B}_L$ such that $v_1w$ and $v_2w$ are factors for any
prefix $w$ of $\wb$. Such a pair is called an infinite LS pair.
\begin{dfn}
    Let $(v_1,v_2)$ be an element of an L-forky set corresponding
    to a fixed point $\ub$ of a morphism $\vp$. The ordered pair
    $((v_1,v_2),\wb)$ is called an \emph{infinite LS pair} if for any
    prefix $w$ of $\wb$ the words $v_1w$ and $v_2w$ are factors of
    $\ub$.

    Further, we define the \emph{$f_{\mathcal{B}_L}$-image} of an infinite LS pair
    $((v_1,v_2),\wb)$ as the infinite LS pair $((v'_1,v'_2),\wb')$, where
    $(v'_1,v'_2) = g_L(v_1,v_2)$ and $\wb' = f_L(v_1,v_2)\vp(\wb)$.
\end{dfn}
Having the $f_{\mathcal{B}_L}$-image of an infinite LS branch, we are
again interested in its $f_{\mathcal{B}_L}$-preimage.
\begin{lem}\label{lem:f_preimage_of_infinite_pairs}
    Let $(\mathcal{A},\vp,a), a  \in \mathcal{A},$ be a circular D0L-system with
    an infinite fixed point $\ub = \vp^\omega(a)$ and let
    $\mathcal{B}_L$ be its L-forky set. Then any infinite LS pair is the $f_{\mathcal{B}_L}$-image of a unique infinite LS pair.
\end{lem}
\begin{proof}
    Let $((v_1,v_2),\wb)$ be an infinite LS pair and let $D$ be a synchronizing
    delay of $\vp$.  Then any prefix of $w$ of length at least $D$
    has the same left-most synchronizing point
    $(w_1,(w_1)^{-1}w)$. Since such $w$ is LS, then $w_1$ must be
    a label of an edge in $\GL{\mathcal{B}_L}_\vp$ whose end-vertex is
    $(v_1,v_2)$ and starting one in $(v'_1,v'_2)$. The infinite word
    $(w_1)^{-1}\wb$ must have a unique $\vp$-preimage $\wb'$.
\end{proof}
Since any infinite LS pair $((v_1,v_2),\wb)$ has an
$f_{\mathcal{B}_L}$-preimage, the in-degree of the vertex
$(v_1,v_2)$ in the graph of left prolongations $\GL{\mathcal{B}_L}_\vp$ must be
at least one.
\begin{cor}
    Let $(\mathcal{A},\vp,a)$ be a circular D0L-system with a fixed point
    $\ub$ and an L-forky set $\mathcal{B}_L$. If $((v_1,v_2),\wb)$ is an
    infinite LS pair then $(v_1,v_2)$ is a vertex of a cycle in
    $\GL{\mathcal{B}_L}_\vp$.
\end{cor}
We know that the number of infinite LS pairs in a fixed point of a
primitive morphism is finite (see
Proposition~\ref{pro:inf_LS_branch}); the following proposition
says that this is true even if we weaken the assumption from
primitive to circular and non-pushy. The proof of the proposition will,
moreover, give us a simple method of how to find all these
infinite LS pairs.
\begin{thm}\label{thm:infinite_branches}
    Let $(\mathcal{A},\vp,a), a  \in \mathcal{A},$ be a circular
    D0L-system with an L-forky set $\mathcal{B}_L$ such that $\vp^{\omega}(a) = \ub$ is
    infinite. Then there exist only a finite number of infinite LS pairs.
\end{thm}
\begin{proof}
    Let $((v_1,v_2),\wb)$ be an infinite LS pair and let
    $(v_1,v_2)$ be a vertex of a cycle in $\GL{\mathcal{B}_L}_\vp$. Let us
    assume that the cycle is of length $k$. Then it is labeled
    by words $f_L(v_1,v_2), f_L(g_L(v_1,v_2)), \ldots, f_L(g_L^{k-1}(v_1,v_2))$ where $f_L(v_1,v_2)$ is the label of the
    edge starting in $(v_1,v_2)$ (see Figure~\ref{fig:proof_of_infinite_branches}).
    \begin{figure}[h]
        \begin{center}\includegraphics{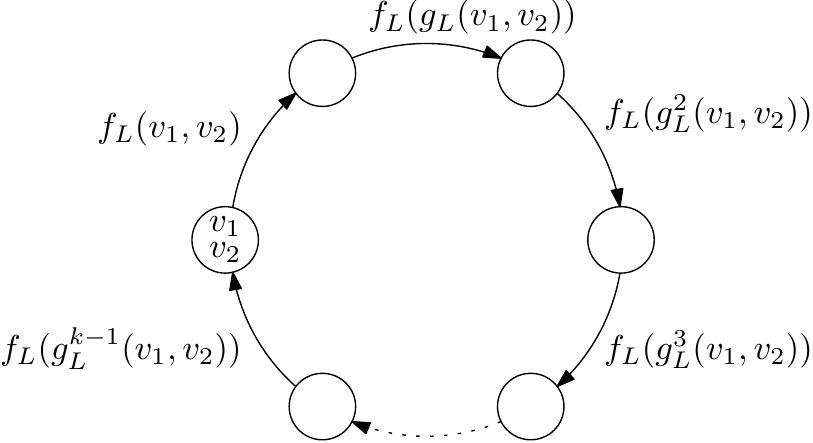}\end{center}
        \caption{The notation from the proof of Theorem~\ref{thm:infinite_branches}.} \label{fig:proof_of_infinite_branches}
    \end{figure}
    We distinguish two cases:\\[2mm]
    $(a)$ At least one of the labels of the cycle is not
    the empty word. Applying $k$ times  Lemma~\ref{lem:f_preimage_of_infinite_pairs} we can find the
    infinite LS pair $((v_1,v_2),\wb')$ such that $w$ is the
    $(f_{\mathcal{B}_L})^{k}$-image of $((v_1,v_2),\wb')$, i.e., 
    $$
        \wb = \underbrace{f_L(g_L^{k -1}(v_1,v_2)) \cdots
        \vp^{k -2}(f_L(g_L(v_1,v_2))\vp^{k
        -1}(f_L(v_1,v_2))}_{\text{denoted by }s}\vp^{k}(\wb') = s\vp^{k}(\wb').
    $$
    Since $\wb'$ can be expressed again as $\wb' =
    s\vp^k(\wb'')$ for some infinite LS pair $((v_1,v_2),\wb'')$, we have
    $$
        \wb = s\vp^k(s)\vp^{2k}(s)\vp^{3k}(\wb'').
    $$
    Continuing in this construction one can prove that
    $s\vp^{k}(s)\cdots\vp^{nk}(s)$ is a prefix of $\wb$ for all $n \in \N$. Therefore, we get
    $$
        \wb = s\vp^{k}(s)\vp^{2k}(s)\vp^{3k}(s)\cdots.
    $$

    We have just shown that exactly one infinite LS pair corresponds to each vertex of the cycle.
    
    \noindent $(b)$ Now assume that all the labels of the cycle are empty words.
    In such a case the $f_{\mathcal{B}_L}$-image coincides with $\vp$-image,
    meaning that $(f_{\mathcal{B}_L})^j$-image of $((v_1,v_2),\wb)$ is
    $(g_L^j(v_1,v_2),\vp^j(\wb))$ for all $j = 1,2,\cdots$. We
    want to prove that $\wb$ must be a periodic point of $\vp$.
    Consider the directed graph whose vertices are the first
    letters of $\vp(b), b \in \mathcal{A},$ and there is an edge from $b$
    to $c$ if $c$ is the first letter of $\vp(b)$. Clearly, the
    first letter of $\wb$, say $b$, must be again a vertex of a
    cycle in this graph. Let $\ell$ be the length of this cycle.
    For reasons analogous to those above the
    $(f_{\mathcal{B}_L})^{j\ell}$-image and $(f_{\mathcal{B}_L})^{j\ell}$-preimage
    of $\wb$ must also begin in $b$. Therefore, $\wb$ contains the
    factor $\vp^{j\ell}(b)$ as a prefix for all $j = 1,2,\ldots$ and this implies that $\wb =
    (\vp^\ell)^\omega(b)$, i.e., $\wb$ is a periodic point of $\vp$.

    Since the number of vertices of $\GL{\mathcal{B}_L}_\vp$ and of
    periodic points is finite, the number of infinite LS pairs
    must be finite as well.
\end{proof}
The previous proof is also a proof of the following
corollary which gives us a method of how to find all infinite LS
branches.
\begin{cor}\label{cor:how_to_find_LS_branches}
    Let $(\mathcal{A},\vp,a),a \in \mathcal{A},$ be a circular D0L-system, $\ub
    = \vp^{\omega}(a)$ infinite with $\mathcal{B}_L$ an L-forky set  and let
    $((v_1,v_2),\wb)$ be an infinite LS pair. Then either $\wb$ is a 
    \emph{periodic point} of $\vp$, i.e.,
    \begin{equation}\label{eq:eq_for_inf_LS_type1}
        \wb = \vp^{\ell }(\wb) \quad \text{for some $\ell \geq 1$},
    \end{equation}
    and $(v_1,v_2)$ is a vertex of a cycle in $\GL{\mathcal{B}_L}_\vp$ labeled by
    $\epsilon$ only, or $\wb = s\vp^{\ell }(s)\vp^{2\ell}(s)\cdots$ is the
    unique solution of the equation
    \begin{equation}\label{eq:eq_for_inf_LS_type2}
        \wb = s\vp^{\ell }(\wb),
    \end{equation}
    where $(v_1,v_2)$ is a vertex of a cycle in $GL_\vp$ containing at
    least one edge with a non-empty label, $\ell$ is the length of this
    cycle and
    \begin{equation}\label{eq:def_of_prefix_s}
        s = f_L(g_L^{\ell -1}(v_1,v_2)) \cdots
        \vp^{\ell -2}(f_L(g_L(v_1,v_2))\vp^{\ell -1}(f_L(v_1,v_2)).
    \end{equation}
\end{cor}
We demonstrate this method on an example morphism.
\begin{exa}
    We consider the morphism
    \begin{equation} \label{eq:ex_subst}
        \vpi{P}: 1 \mapsto 1211, 2 \mapsto 311, 3 \mapsto 2412, 4 \mapsto 435, 5 \mapsto 534
    \end{equation}
    with $\ub = \vpi{P}^\omega(1)$. This morphism is suffix- and prefix-free and
    so the set of all unordered pairs of distinct letters is L-forky. The graph
    of left prolongations is in Figure~\ref{fig:GL_example}.
    \begin{figure}
        \begin{center}\includegraphics{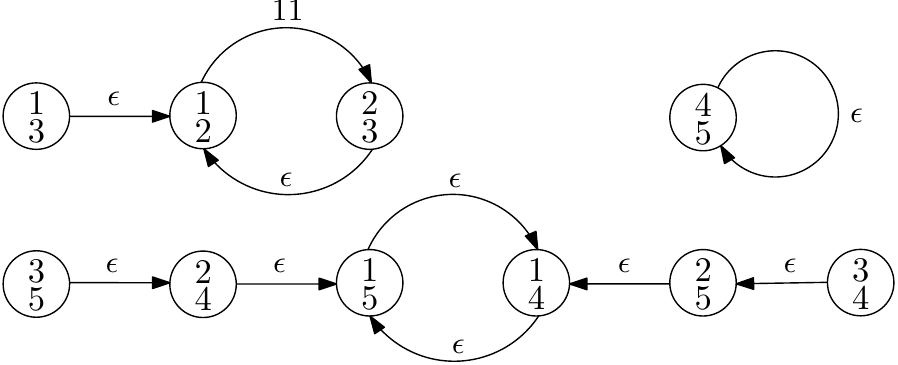}\end{center}
        \caption{The graph $GL^{(\mathcal{B}_1)}_{\vpi{P}}$ for morphism defined
        by~\eqref{eq:ex_subst}, $\mathcal{B}_1$ is the set of all unordered
        pairs of letters.}
        \label{fig:GL_example}
    \end{figure}
    The morphism $\vpi{P}$ has five periodic points
    $$
        \vpi{P}^\omega(1),\vpi{P}^\omega(4),\vpi{P}^\omega(5),(\vpi{P}^2)^\omega(2),(\vpi{P}^2)^\omega(3).
    $$
    It is easy to show that
    \begin{multline*}
        \Le(1) = \{1,2,3,4,5\}, \Le(2) = \{1,4,5\}, \Le(3) = \{1,4,5\},\\
        \Le(4) = \{1,2,3\}, \Le(5) = \{1,2,3\}.
    \end{multline*}
    Looking at the graph of left prolongations depicted in
    Figure~\ref{fig:GL_example}, we see that
    $\vpi{P}^\omega(4)$ and $\vpi{P}^\omega(5)$ are not infinite LS branches as none
    of the vertices $(1,2), (2,3)$ and $(1,3)$ is a vertex of a cycle
    labeled by $\epsilon$ only. Hence, only
    $\vpi{P}^\omega(1),(\vpi{P}^2)^\omega(2),(\vpi{P}^2)^\omega(3)$ are infinite
    LS branches with left extensions $1,4,5$.

    As for infinite LS branches corresponding to
    Equation~\eqref{eq:eq_for_inf_LS_type2}, in the case of our
    example, there is only one cycle which is not labeled by the empty word:
    the cycle between vertices $(1,2)$ and $(2,3)$. There are two (= the
    length of the cycle) equations corresponding to this cycle
    $$
        \wb = \vpi{P}(11)\vpi{P}^2(\wb) \quad \text{and} \quad \wb =
        11\vpi{P}^2(\wb).
    $$
    They give us two infinite LS branches
    \begin{equation*}
        \begin{split}
            & \vpi{P}(11)\vpi{P}^3(11)\vpi{P}^5(11)\cdots,\\
            & 11\vpi{P}^2(11)\vpi{P}^4(11)\cdots,
        \end{split}
    \end{equation*}
    the former having left extensions $1$ and $2$ and the latter $2$
    and $3$.
\end{exa}

\section{Assumptions and a connection with the critical exponent}
\label{sec:assumptions}

The method of how to generate all BS factors of a given D0L-system we have
described above bears on two facts: There exists L- and R-forky sets and the
number of initial BS triplets is finite. The former was proved for circular and
non-pushy systems and the later for circular ones only. Are these assumptions
necessary or can they be weakened?

Let us take a morphism $0 \mapsto 001, 1 \mapsto 1$ which is pushy and circular
and its fixed point $\ub$ starting in $0$. If we try to construct an L-forky set as it
has been defined in this paper, we will find out that it is not possible. A
natural cantidate for veritices of the graph of left prolongations are pairs
$(0,1^n)$ with $n \in \N$, but any finite set of such pairs does not satisfy the property
$(iv)$ of Definition~\ref{dfn:L_forky}. So it seems that to assume the morphism 
being non-pushy is inevitable for existence of forky
sets. However, if we relax the definition and enable the pairs of factors to be
infintely long, we can find something like L-forky set even for this morphism:
Define a directed graph of left prolongations such that it has only one vertex
$(0,\cdots 111)$ and one loop on this vertex with label 1 and a directed graph
of right prolongations with one vertex $(0,1)$ and a loop on it with empty label,
then all BS factors in $\ub$ are the $f$-images of BS-triplet $((0,\cdots
111),0,(0,1))$, namely $0, 1\vp(0), 1\vp(1)\vp^2(0), \ldots$

Now, consider a morphism $0 \mapsto 001, 1 \mapsto 11$
and its fixed point $\ub$ starting in $0$. This morphism is non-pushy and
non-circular. In this case, L- and R-forky sets exists, we can simply take
$\{(0,1)\}$ and there is only one (nonempty) initial $BS$-triplet with no
\emph{BS-synchronizing} point $((0,1),0,(0,1))$. Its $f$-images again reed $0,
1\vp(0), 1\vp(1)\vp^2(0), \ldots$. In fact, to prove that the set of initial 
BS-triplets is finite, we need to know only that there is not infinite numer of
BS-triplets without BS-synchronizing point and it seems to be true even for
non-circular morphisms.

All considered morphisms for which our method does not work (or is not proved to
work) have an infinite critical exponent. The following theorem says this is not
a misleading observation but a general rule.
\begin{thm}
	Let $G = ({\A}, \vp, w)$ be a D0L-system.
	Then the critical exponent of ${\Lan}(G)$ is finite if and only if $G$ is
	circular and non-pushy.
\end{thm}

\begin{proof}
$(\Rightarrow)$: Circularity follows from Theorem~\ref{thm:mignosi_seebold}. $G$
being pushy is in contradiction with Lemma~\ref{lem:pushy_systems}, thus, it is non-pushy.

\vspace{\baselineskip}
\noindent $(\Leftarrow)$: Suppose the critical exponent of $\mathcal{L}(G)$ is
infinite and that $G$ is circular and non-pushy. According to
Theorem~\ref{thm:strongly_repetitive}, there exists a non-empty factor $v \in \mathcal{L}(G)$ such that for all $n \in \N$, $v^n \in \mathcal{L}(G)$. Take the shortest factor $v$ having such property.
Since $G$ is circular, there exists a finite synchronizing delay $D$.
Take $N \in \N$ such that $|v^N| \geq D$.
Then $v^N$ contains a synchronizing point, i.e., $v^N = v_1 |_s v_2$.
It is clear that $v^{N+1}$ contains at least two synchronizing points, i.e.,
$v^{N+1} = v_1 |_s v_2 v = v v_1 |_s v_2$. In general, $v^{N+k}$ contains $k+1$
synchronizing points at fixed distances equal to $|v|$. Since $\vp$ is injective, it implies that there exists a unique $z \in \mathcal{L}(G)$ such that $v^{N+k} = p \vp(z^{k})s$ 
(for some factors $p$ and $s$) and $z^k \in \mathcal{L}(G)$ for all $k \geq 0$.
According to the choice of $v$, it is clear that $|\vp(z)| = |z| = |v|$.
Denote by ${\mathcal{L}}_1(z)$  the set of letters occurring in $z$.
It is clear that $\vp({\mathcal{L}}_1(z)) = {\mathcal{L}}_1(v)$ and $\forall a \in
{\mathcal{L}}_1(z)$ we have $|\vp(a)| = 1$.

We can now repeat the process: take the factor $z$ to play the role of factor $v$.
Thus, we can find an infinite sequence of factors $z_0 = z, z_1, z_2 \ldots $ such that
$\vp({\mathcal{L}}_1(z_{k+1})) = {\mathcal{L}}_1(z_{k})$ and $|z_k| = |z|$ for all $k \geq
0$. Since $\A$ is finite, it is clear that there exists integers $m \neq \ell$
such that ${\mathcal{L}}_1(z_m) = {\mathcal{L}}_1(z_\ell)$. This implies that for all $k$ the
factor $z_k$ is composed of letters of rank zero. This is a contradiction with $G$ being non-pushy.
\end{proof}

\section{Conclusion}

The tool we have introduced in this paper enables to construct an algorithm
which can find all BS factors in a given circular non-pushy D0L-system so that
it produces the graphs of prolongations and the set of initial BS factors -- its
slightly simplified version was implemented by \v{S}t\v{e}p\'{a}n Starosta
using SAGE. The scatch of the algorithm is as follows:
\begin{enumerate}
  \item Decide whether the input D0L-system is strongly repetitive using the
  algorithm from \cite{Ehrenfeucht1983}. If it is, then by the previous theorem
  and Theorem~\ref{thm:strongly_repetitive} the D0L-system is non-circular or
  pushy and our method does not work. If it is not, proceed with to steps.
  \item Construct L- and R-forky sets: the details of the construction are a bit
  technical but the basic idea is the same we used in
  Example~\ref{exa:vp_S_set_B}.
  \item Find all initial BS triplets without any BS-synchronizing point. The
  fact that the system is circular ensures the algorithm stops after a finite
  number of steps.
\end{enumerate}

\section{Acknowledgement}

We acknowledge financial support by the Czech Science Foundation grant
201/09/0584,  and by the grants MSM6840770039 and LC06002 of the Ministry of
Education, Youth, and Sports of the Czech Republic. The author also thanks
\v{S}t\v{e}p\'{a}n Starosta and Edita Pelantov\'{a} for their suggestions and
the fruitful discussions.





\bibliographystyle{plain}
\bibliography{publications}

\end{document}